\documentclass[11pt]{article}
\usepackage{amsmath,amscd,amsthm,amssymb,cleveref,cite,xy}
\usepackage{url}

\input xy \xyoption{all}
\CompileMatrices
\newtheorem{theorem}{Theorem}

\newtheorem{remark}[theorem]{Remark}
\newtheorem{proposition}[theorem]{Proposition}
\newtheorem{corollary}[theorem]{Corollary}
\newtheorem{definition}[theorem]{Definition}

\renewcommand{\int}{\wr}
\newcommand{\sO}{\mathcal{O}}
\newcommand{\Z}{\mathbb{Z}}
\newcommand{\Kinv}{K^{\mathrm{inv}}}
\newcommand{\Kh}{\mathbb{K}}
\newcommand{\F}{\mathbb{F}}
\newcommand{\Khinv}{\mathbb{K}^{\mathrm{inv}}}
\newcommand{\Q}{\mathbb{Q}}
\newcommand{\C}{\mathcal{C}}
\newcommand{\D}{\mathcal{D}}

\renewcommand{\F}{{\mathbb F}}

\newcommand{\n}{\mathfrak{n}}
\newcommand{\Hom}{\mathrm{Hom}}
\newcommand{\Fun}{\mathrm{Fun}}
\renewcommand{\ker}{\mathrm{ker}}
\newcommand{\Sp}{\mathrm{Sp}}
\newcommand{\Sets}{\mathrm{Sets}}
\newcommand{\Ring}{\mathrm{Ring}^{\mathrm{nu,h}}}
\newcommand{\RingG}{\mathrm{Ring}^{\mathrm{nu,h,G}}}
\newcommand{\Spec}{\mathrm{Spec}}

\newcommand{\Proj}{\mathbb{P}\mathrm{roj}}
\newcommand{\id}{\mathrm{id}}

\begin{document}

\author{Niko Naumann, Charanya Ravi
\thanks{We thank Georg Tamme for useful conversation, and for catching a mistake in the proof of Proposition \ref{prop:geometric_rigidity}. Both authors were supported
through the SFB 1085, Higher Invariants, Regensburg.}}
\date{\today}
\title{Rigidity in equivariant algebraic $K$-theory}
\maketitle

\begin{abstract} If $(R,I)$ is a henselian pair with an action of a finite group
$G$ and $n\ge 1$ is an integer coprime to $|G|$ and such that $n\cdot |G|\in R^*$, then the reduction map of mod-$n$
equivariant $K$-theory spectra \[ K^G(R)/n\stackrel{\simeq}{\longrightarrow} K^G(R/I)/n\]
is an equivalence. We prove this by revisiting the recent proof of non-equivariant rigidity by Clausen, Mathew, and Morrow. 
\end{abstract}

\section{Introduction and statement of result}
Rigidity is a fundamental feature of algebraic $K$-theory with finite coefficients 
which was established by Suslin \cite{suslin1} for extensions of algebraically closed fields, and by
Gabber and Gillet-Thomason \cite{GT} for geometric henselian local rings. In \cite{gabber}, inspired by previous results of Suslin \cite{suslin2}
for henselian valuation rings of dimension one, Gabber proved a rigidity theorem for algebraic $K$-theory with finite coefficients for general henselian pairs:

\begin{theorem}[Gabber] If $(R,I)$ is a henselian pair and $n\ge 1$ is an integer such that
$n\in R^*$, then \[ K(R)/n\stackrel{\simeq}{\longrightarrow}K(R/I)/n\] is an equivalence.
\end{theorem} \label{thm:Gabber}

In all these results, the coefficients are assumed to be coprime to the characteristic. In \cite{CMM}, the authors established the
most comprehensive rigidity statement to date addressing the case of coefficients not necessarily coprime to the characteristic. To formulate it, we denote by $\Kinv$ the fiber 
of the cyclotomic trace $K\longrightarrow TC$. Then their result \cite[Theorem A]{CMM}
reads

\begin{theorem}[Clausen, Mathew, Morrow] If (R,I) is a henselian pair and $n\ge 1$ is an integer, then the reduction map \[ \Kinv(R)/n\stackrel{\simeq}{\longrightarrow} \Kinv(R/I)/n\] is an equivalence.
\end{theorem}

The purpose of the present note is to generalize this result to an equivariant situation
for an action of a finite, abstract group $G$.

Given a commutative ring $R$ with an action of $G$, there is associated the twisted group ring
$R\int G$, see Section \ref{sec:pscoh} for a reminder. Our main theorem is

\begin{theorem}\label{thm:main}
If the finite group $G$ acts on the henselian pair $(R,I)$, $|G|\in R^*$, and $n\ge 1$ is an integer coprime to $|G|$, then the reduction map
\[ \Kinv(R\int G)/n\stackrel{\simeq}{\longrightarrow} \Kinv((R/I)\int G)/n\]
is an equivalence.
\end{theorem}

The more traditional invariant in equivariant algebraic $K$-theory is the spectrum $K^G(R)$, defined to be the connective $K$-theory of the exact category of finitely generated projective $R$-modules together with a semi-linear $G$-action. We deduce the next result about this.

\begin{corollary}\label{cor:main} Assume in the situation of Theorem \ref{thm:main} that in addition
$n\in R^*$ holds. Then the reduction map
\[ K^G(R)/n\stackrel{\simeq}{\longrightarrow} K^G(R/I)/n\]
is an equivalence.
\end{corollary}

\begin{proof} Since $n\in R^*$, the $TC$-term in the definition of $\Kinv(R\int G)$ vanishes mod $n$, i.e.
$\Kinv(R\int G)/n\simeq K(R\int G)/n$. Since $|G|\in R^*$, a finitely generated projective left 
$R\int G$-modules is the same thing as a finitely generated projective $R$-module with a semi-linear
$G$-action, hence $K(R\int G)\simeq K^G(R)$; and similarly with $R$ replaced with $R/I$.
\end{proof}

\begin{remark} The appearance of $R\int G$ might seem a bit spurious since all our results
assume $|G|\in R^*$ which forces $K(R\int G)\simeq K^G(R)$ (and similarly for $TC$).
Since however several of the intermediate results work without assuming that $|G|\in R^*$,
we decided to phrase things in terms of $R\int G$.
\end{remark}

Corollary \ref{cor:main} is a generalization of Theorem \ref{thm:Gabber} for equivariant algebraic $K$-theory.
Rigidity results for equivariant algebraic $K$-theory have been previously studied for henselian local rings with trivial group actions
(but for more general algebraic groups) in \cite{YO} and \cite{krishna} and in \cite{YO} and \cite{tabuada} for extensions of algebraically closed fields
and extensions of separably closed fields, respectively. In \cite{HRO}, Corollary \ref{cor:main} was proved in the geometric case, assuming that $G$ is abelian and that $k$ contains $|G|$-th roots of unity.

The proofs of our results are direct generalizations of those of \cite{CMM}.
We made an effort to make this paper reasonably self-contained, which results in repeating some arguments from \cite{CMM}.\\
We conclude the introduction with an overview of the sections.
In Section \ref{sec:pscoh}, we establish the equivariant generalization of the key finiteness property,
called pseudocoherence, isolated in \cite{CMM}. This allows to generalize equivariant rigidity from 
certain nice geometric situations to general henselian pairs. In Section \ref{sec:special_case}
we establish a sufficient supply of equivariant rigidity in nice situations (see Proposition \ref{prop:geometric_rigidity})
by combining the non-equivariant result with decomposition results of Vistoli and 
Tabuada-Van den Bergh. Section \ref{sec:technicalities} collects further technical results. 
The final Section \ref{sec:proof_of_main} assembles the pieces into a proof of Theorem \ref{thm:main}.

 \section{$G$-projective pseudocoherence}\label{sec:pscoh}

The aim of this section is to establish the equivariant generalizations of the finiteness properties
\cite[Propositions 4.21, 4.25]{CMM} of algebraic $K$-theory and of topological cyclic homology with finite coefficients.
Fix a finite group $G$ throughout.

Let $R$ be a commutative ring and $I \subset R$ an ideal.
Recall that the pair $(R,I)$ is called a henselian pair if for every $f(t) \in R[t]$, $\bar{a} \in R/I$, such that $\bar{a}$ is a simple root of $\bar{f}(t) \in (R/I)[t]$, there exists $a \in R$ such that $a \mapsto \bar{a}$ and $f(a) = 0$.
By a result of Gabber \cite[Corollary 1]{gabber}, the property of $(R,I)$ being a henselian pair depends only on the ideal $I$, regarded as a nonunital ring, and not on $R$. 
We now briefly recall the definition of nonunital henselian algebras. For a detailed discussion see \cite[Section 3]{CMM}.

For a commutative ring $R$, a nonunital $R$-algebra is an $R$-module $I$ endowed with a multiplication $I \otimes_R I \to I$
which is associative and commutative. A non-unital $R$-algebra $I$ is said to be henselian if 
for every $n \geq 0$ and every $g(t) \in I[t]$ of degree at most $n$, the polynomial $$f(t) = t (1+t)^n + g(t)$$
has a (necessarily unique) root in $I$.
Let $\Ring_R$ denote the category of non-unital, henselian $R$-algebras.

\begin{definition}
We denote by $\RingG_R$ the category of $G$-objects in $\Ring_R$.
\end{definition}
To ease the notation, we will abbreviate $\RingG:=\RingG_\Z$.

It is observed in \cite[Remark 3.10]{CMM} that the category $\Ring_R$ is bi-complete, and that the forgetful functor \[R:\Ring_R\longrightarrow\Sets\] to sets
is a conservative right adjoint which commutes with sifted colimits.\footnote{equivalently, as the categories are discrete, it
commutes with filtered colimits and split co-equalizers.} 

Denoting by \[F_R:\Sets\longrightarrow\Ring_R\] its left-adjoint, this is remarked to imply that the subcategory $\left(\Ring_R\right)_\Sigma\subseteq\Ring_R$ of compact projective objects is the idempotent 
completion of the full subcategory spanned by the free objects \[ F_R(n):=F_R(\{ 1, \ldots, n\})\,\, (n\ge 0).\]
Moreover, $F_R(n)$ is 
identified in \cite[Example 3.9]{CMM} as the ideal generated by the variables $X_1,\ldots, X_n$ in the $R$-algebra given by the henselization of 
$R[X_1,\ldots, X_n]$ along the ideal $(X_1,\ldots, X_n)$.

This generalizes to the equivariant setting as follows: The category $\RingG$ is bi-complete and the forgetful functor 
\[R':\RingG_R\longrightarrow\Ring_R\]
is a conservative right-adjoint which commutes with all colimits. This is clear by thinking of
$\RingG$ as the category of presheaves on $G$ with values in $\Ring$. Consequently, denoting the left-adjoint of $R'$ by
\[F'_R:\Ring_R\longrightarrow\RingG_R,\] and by $F''_R:=F'_R\circ F_R$, the subcategory $\left(\RingG_R\right)_\Sigma\subseteq\RingG_R$ of compact projective objects is the idempotent 
completion of the full subcategory spanned by the free objects $F''_R(n):=F''_R(\{ 1, \ldots, n\})$ $(n\ge 0)$.
These be identified explicitly:

\begin{proposition}\label{prop:equivariant_generators}
For every $n\ge 0$, $F''_R(n)$ is the ideal generated by the variables $X_{\sigma,i}$ $(\sigma\in G, 1\leq i\leq n)$  in the $R$-algebra given by the henselization of the polynomial $R$-algebra $R[X_{\sigma,i} \mid \sigma\in G, 1\leq i\leq n]$ along the ideal $(X_{\sigma,i})$, and $G$-action determined by $\sigma(x_{\tau,i})=x_{\sigma\tau,i}$.
\end{proposition}

Said a bit more invariantly, $F''_R(n)$ is the henselization along the origin of 
the affine $R$-space afforded by the direct sum of $n$ copies of the regular representation of $G$ over $R$.

\begin{proof}[Proof of Proposition \ref{prop:equivariant_generators}]
Since henselization is a left-adjoint, it suffices to see the analogous statement before 
henselization. Then using the equivalence between non-unital $R$-algebras and augmented
$R$-algebras, the claim follows because the augmented $R$-algebra with $G$-action
$R[X_{\sigma,i} \mid \sigma\in G, 1\leq i\leq n]$ has the required mapping property.
\end{proof}

For every $N\ge 1$, we will denote by \[ [N]:F''_R(n)\longrightarrow F''_R(n)\]
the ``multiplication-by-$N$-map", namely the unique map in $\RingG_R$ which, under 
the identification of Proposition \ref{prop:equivariant_generators}, maps every $X_{\sigma,i}$ to
$NX_{\sigma,i}$.

\begin{proposition}\label{prop:equivarint_gabber_colim}
For fixed $M\ge 1$ and $n\ge 0$, we have an isomorphism in $\RingG_{\Z\left[ \frac{1}{M}\right]}$
\[ \mathrm{colim}_{(N,M)=1} F''_{\Z\left[ \frac{1}{M}\right]}(n)\simeq F''_\mathbb{Q}(n),\]
the (filtered) colimit being taken along the multiplication maps $[N]$ for all $N$
coprime to $M$, partially ordered by divisibility.
\end{proposition}

\begin{proof} This is proved exactly as in the special case $M=1, G=\{ e\}$ which is due to Gabber (see \cite[Corollary 3.20]{CMM}). We leave the details to the reader.
\end{proof}

Recall the twisted group ring (e.g. \cite[\S 28]{curtis_reiner}): If $R$ is a commutative ring with a (left) $G$-action, then 
the twisted group ring $R\int G$ is the finite free $R$-module on the set $\left\{e_\sigma\colon\sigma\in G\right\}$ 
with multiplication determined by $(re_\sigma)(r'e_\tau)=r\sigma(r')e_{\sigma\tau}$. This construction
is functorial in $R$.
It is rigged such that the datum of a left $R\int G$-module is equivalent to the datum of an $R$-module together with a {\em semi-}linear $G$-action. Observe that when the $G$-action on $R$
is trivial, this construction gives the usual group ring, i.e. $R\int G=R[G]$ in this case.

For an associative, unital ring $A$, we will denote by $K(A)$ the connective $K$-theory spectrum of 
the category of finitely generated projective left $A$-modules, cf. \cite{quillen}. Given any 
$I\in\RingG$, we denote by $\Z\ltimes I$ the ring with $G$-action obtained from $I$ by adjoining a unit 
(necessarily with trivial $G$-action). The augmentation $\Z\ltimes I\longrightarrow \Z$
is $G$-equivariant and thus induces an augmentation $p:(\Z\ltimes I)\int G\longrightarrow \Z[G]$.
We will need the following equivariant generalization of \cite[Lemma 4.20]{CMM}.

\begin{proposition}\label{prop:K_0} Given $I\in\RingG$, the map $p^*: K_0\left((\Z\ltimes I)\int G\right)\stackrel{\simeq}{\longrightarrow}K_0(\Z[G])$ is an isomorphism.
\end{proposition}

\begin{proof} This is a special case of Proposition \ref{prop:eq_rig_K0}.
\end{proof}

We denote by $\Sp$ the $\infty$-category of spectra. We recall from \cite[Definition 4.4]{CMM} the notions of perfectness and pseudocoherence
of spectrum-valued functors on a category relative to a subcategory:
Given a small full subcategory $\D$ of a locally small category $\C$, a functor $F: \C \to \Sp$ is called $\D$-perfect if $F$
belongs to the thick subcategory generated by the functors $\{ \Sigma^\infty_+ \Hom_\C (D, \cdot)\,\mid\, D\in{\mathcal D}\}$ in the presentable, stable $\infty$-category
$\Fun(\C, \Sp)$. A functor $F \in \Fun(\C, \Sp)$ is said to be
$\D$-pseudocoherent if for each $n \in \Z$, there exists a $\D$-perfect functor $F_n$ and a map $F_n \to F$ such that
$\tau_{\leq n}F_n(C) \to \tau_{\leq n}F(C)$ is an equivalence for all $C \in \C$.
 In the particular case
when $\D=\left(\Ring\right)_\Sigma\subseteq\C=\Ring$, $F$ is called projectively pseudocoherent, see \cite[Definition 4.12, (2)]{CMM}. We pose the immediate equivariant generalization of this as a definition.

\begin{definition}
A functor $F:\RingG_R\longrightarrow\Sp$ is called $G$-projectively pseudocoherent ($G$-pscoh for short),
if it is $\left(\RingG_R\right)_\Sigma$-pseudocoherent.
\end{definition}

Our first aim then is to establish the following generalization of \cite[Proposition 4.21]{CMM}. 

\begin{proposition}\label{prop:K_G_pscoh}
The functor $\RingG\longrightarrow\Sp$, $I\mapsto K((\Z\ltimes I)\int G)$ is $G$-pscoh.
\end{proposition}

\begin{proof}
Using the fiber sequence of functors 
\[ \tau_{\ge 1}K((\Z\ltimes (-))\int G)\longrightarrow K((\Z\ltimes (-))\int G)\longrightarrow
\tau_{\leq 0}K((\Z\ltimes (-))\int G)=K_0((\Z\ltimes (-))\int G)\]
and the fact that $G$-pscoh functors form a thick subcategory \cite[Proposition 4.8, (1)]{CMM}, it suffices to see separately
the $G$-projective pseudocoherence of $\tau_{\ge 1}K((\Z\ltimes (-))\int G)$ and of
$K_0((\Z\ltimes (-))\int G)$.

For the latter, Proposition \ref{prop:K_0} yields an isomorphism $K_0((\Z\ltimes (-))\int G)\simeq K_0(\Z[G])$
to the constant functor with value the finitely generated abelian group $K_0(\Z[G])$, see
\cite[Theorem 2.2.1]{kuku_book}. This settles the claim for this term.

To see that the other term is $G$-pscoh, we use the criterion \cite[Proposition 4.10 and Proposition 4.11]{CMM} to reduce to seeing that 
the functor
\[ H\Z\otimes\Sigma_+^\infty\Omega^\infty \tau_{\ge 1}K((\Z\ltimes (-))\int G)\]
is $G$-pscoh. It is well known (see \cite[chapter IV, \S 1]{weibel_k_book}) that 
this functor is equivalent to $C_*(BGL((\Z\ltimes (-))\int G);\Z)$, the complex of integral chains on the
classifying space of the infinite general linear group. We now use homology stability as given by
\cite[Theorem in section 4.11]{van_kallen} for the associative ring $A(-):=(\Z\ltimes (-))\int G$. To do so, we need to see that the stable range of $A(-)$ is bounded independently of the argument
$-\in\RingG$. Firstly, it is easy to see that dividing out a radical ideal does not change the stable range (cf. \cite[p. 32]{lam_bass_work} and \cite[chapter I, ex. 1.12, (v)]{weibel_k_book}), and at the beginning of the proof of Proposition \ref{prop:eq_rig_K0} we will see that $(-)\int G$ is a radical ideal
in $A(-)$ with quotient ring $\Z[G]$. This already gives the independence of the stable range of $A(-)$ of the argument $(-)$, and since $\Z[G]$ is finite over its central subring $\Z$, this is bounded by (in fact, equal to) the stable range of $\Z$ (according to Bass's stable range theorem \cite[Chapter V, Theorem 3.5]{bass}). We conclude that for every $n\ge 1$ the obvious map on truncations
\[ \tau_{\leq n}C_*(BGL_{2n+1}((\Z\ltimes (-))\int G);\Z)\longrightarrow \tau_{\leq n}C_*(BGL((\Z\ltimes (-))\int G);\Z)\]
is an equivalence. Renaming indices, this reduces us to seeing that for a fixed $n\ge 1$, the functor
\[ C_*(BGL_n((\Z\ltimes (-))\int G);\Z)\]
is $G$-pscoh. There is a short exact sequence of groups
\[ 1\longrightarrow X(-)\longrightarrow GL_n((\Z\ltimes (-))\int G)\stackrel{\pi}{\longrightarrow} GL_n(\Z[G])\longrightarrow 1,\]
defining $X(-)$. \footnote{To see that $\pi$ is onto, recall that the augmentation $\Z\ltimes (-))\int G\longrightarrow\Z[G]$ is {\em split} surjective.}
This gives an equivalence
\[ C_*(BGL_n((\Z\ltimes (-))\int G);\Z)\simeq\left( C_*(BX(-);\Z) \right)_{h(GL_n(\Z[G]))}.\]
To conclude the argument exactly as in the proof of \cite[Proposition 4.19]{CMM}, it remains to establish that, firstly, the functor 
$C_*(BX(-);\Z)$ is $G$-pscoh and that, secondly, there is a finite index normal subgroup $N\subseteq 
GL_n(\Z[G])$ such that its classifying space $BN$ is equivalent to a finite CW-complex. The first claim follows as in
{\em loc.cit.‚}, because
$X(I)\simeq I^{\mid G\mid\cdot n^2}$ (as sets), and the second claim follows from work of Borel and Serre,
specifically \cite[section 2.4, Th\'eor\`eme 4 and section 1.5, Proposition 10]{serre}, if we can show that 
$GL_n(\Z[G])$ is an arithmetic subgroup of a suitable reductive group ${\mathcal G}$ over $\Q$. Indeed, one can take for ${\mathcal G}$ the group of units of the $\Q$-algebra $M_n(\Q[G])$:
It is clear that \[ GL_n(\Z[G])\subseteq {\mathcal G}(\Q)=GL_n(\Q[G])\] is an arithmetic
subgroup, and since $\Q[G]\otimes_\Q\mathbb{C}\simeq\mathbb{C}[G]$ is a product of full matrix rings over
$\mathbb{C}$, the group ${\mathcal G}\otimes_\Q \mathbb{C}$ is a finite product of various $GL_{i,\mathbb{C}}$'s, and
hence is (connected and) reductive.
\end{proof}

The following generalization of \cite[Proposition 4.25]{CMM} is even more immediate.

\begin{proposition}\label{prop:TC_G_pscoh} For every prime $p$, the functor $\RingG\longrightarrow\Sp$, $I\mapsto
TC((\Z\ltimes I)\int G)/p$ is $G$-pscoh.
\end{proposition}

\begin{proof} This is identical to {\em loc. cit.}, and we leave the details to
the reader. Recall at least that the core part of the argument, namely \cite[Proposition 2.19]{CMM}, is a result about
$TC(-)/p$ considered on the category of cyclotomic spectra, which 
applies equally well to the case at hand.
\end{proof}

Recall that we write $\Kinv$ for the fiber of the cyclotomic trace $K\longrightarrow TC$.
We introduce a relative term $\Kinv((\Z\ltimes I)\int G,I\int G)$ to sit in a fiber sequence
\[ \Kinv((\Z\ltimes I)\int G,I\int G)\longrightarrow \Kinv((\Z\ltimes I)\int G)\longrightarrow
\Kinv(\Z\int G)=\Kinv(\Z[G]).\]
Combining Propositions \ref{prop:K_G_pscoh} and \ref{prop:TC_G_pscoh} yields the following, which is the finiteness result
to be used in the proof of Theorem \ref{thm:main}.

\begin{proposition}\label{prop:G_pscoh_final} For every prime $p$, the functor
\[ \RingG\longrightarrow\Sp\, , \, I \mapsto \Kinv((\Z\ltimes I)\int G, I\int G)/p\]
is $G$-pscoh.
\end{proposition} 

\section{A geometric special case}\label{sec:special_case}

The purpose of this section is to establish a geometric special case of our main
result Theorem \ref{thm:main}. This equivariant rigidity result will follow from its non-equivariant special
case \cite[Theorem A]{CMM} together with decomposition results of Vistoli and Tabuada-Van den Bergh \cite{tabuada_vdB}.
To formulate it, fix for the rest of this section a finite group $G$, a field $k$ of characteristic not dividing $|G|$ \footnote{By convention, this condition is satisfied if $k$ is of characteristic zero.} and a prime $p$ not dividing $|G|$ (but possibly equal to the characteristic of $k$). Let $X$ be an affine, smooth $k$-algebra 
with a $G$-action and assume given a rational point $x\in X(k)$ fixed by $G$. Then $G$ acts canonically on the henselization $\sO^h_{X,x}$ of the local ring $\sO_{X,x}$, and the canonical map $\pi\colon\sO^h_{X,x}
\longrightarrow k$ to the residue field is $G$-equivariant (for $k$ endowed with the trivial $G$-action). Hence it induces a map on twisted group rings $\sO^h_{X,x} \int G\longrightarrow k\int G=k[G]$.
The result then is the following.

\begin{proposition}\label{prop:geometric_rigidity}
In the above situation, the map induced by $\pi$
\begin{equation}\label{eq:claim} \Kinv(\sO^h_{X,x} \int G)/p\stackrel{\simeq}{\longrightarrow} \Kinv(k[G])/p
\end{equation}
is an equivalence.
\end{proposition}




\begin{proof} We start by setting the stage to apply 
\cite{tabuada_vdB}. We denote by $E:=\pi_*\left(\Kinv(-)/p\right)$, and observe that this is an
additive invariant taking values in $\Z\left[\frac{1}{|G|}\right]$-modules and commuting with filtered colimits: For algebraic $K$-theory, this is classical and for $TC(-)/p$ it follows from \cite[Theorem 2.7]{CMM}.
Now, \cite[Remark 1.3, ii) and iii)]{tabuada_vdB} implies that 

\begin{equation}\label{eq:decomp_map} E([X/G])\stackrel{\simeq}{\longrightarrow}\left( \bigoplus\limits_{\sigma\subseteq G\mbox{ {\small cyclic}}} \widetilde{E}\left( X^\sigma\times\Spec(k[\sigma])\right) \right)^G,
\end{equation}
where
$X^\sigma\subseteq X$ is the subscheme fixed by $\sigma$, and $\widetilde{E}$
refers to a certain functorially defined direct summand of $E$ (depending on $\sigma$).
Since we will not require knowledge of the exact shape of that summand, we will not review its definition here.

We observe that the $G$-fixed point $x\in X(k)$ determines a map
 \[\bar{x}:[\Spec(k)/G]\longrightarrow [X/G]\] such that $E(\bar{x})$ participates in a commutative diagram
\begin{equation}\label{eq:diagram1}
\xymatrix{ E([X/G])\ar[rr]^-{\simeq\,\cref{eq:decomp_map}}\ar[dd]^{E(\bar{x})} & & \ar[dd]^{\oplus_\sigma\widetilde{E}(x_\sigma\times\mathrm{id})}\left( \bigoplus\limits_{\sigma\subseteq G\mbox{ {\small cyclic}}} \widetilde{E}\left( X^\sigma\times\Spec(k[\sigma])\right) \right)^G\, \\
 & &  \\E([\Spec(k)/G])\ar[rr]^-{\simeq} & &\,\left( \bigoplus\limits_{\sigma\subseteq G\mbox{ {\small cyclic}}} \widetilde{E}\left(\Spec(k[\sigma])\right) \right)^G,}
\end{equation}
where $x_\sigma$ denotes the unique factorization of $x$ through $X^\sigma\subseteq X$.

Next we want to pass to henselizations. To do this, we observe that everywhere in the above argument, one can replace $([X/G],x)$ with a pointed \'etale neighborhood $(Y,y)$ such that $\kappa(x)\stackrel{\simeq}{\longrightarrow}\kappa(y)$ is an isomorphism on residue-fields. We obtain 
a commutative diagram generalizing \cref{eq:diagram1}
\begin{equation}\label{eq:diagram2}
\xymatrix{ E([Y/G])\ar[rr]^-{\simeq}\ar[dd]^{E(\bar{y})} & & \left( \bigoplus\limits_{\sigma\subseteq G\mbox{ {\small cyclic}}} \widetilde{E}\left( Y^\sigma\times\Spec(k[\sigma])\right) \right)^G \ar[dd]^{\oplus_\sigma \widetilde{E}(y_\sigma\times\id)}\\
 & &  \\
E([\Spec(k)/G])\ar[rr]^-{\simeq} & & \,\left( \bigoplus\limits_{\sigma\subseteq G\mbox{ {\small cyclic}}} \widetilde{E}\left(\Spec(k[\sigma])\right) \right)^G .  }
\end{equation}
Passing to the filtered colimit of all such $(Y,y)$ and recalling that henselization commutes
with the closed immersions $X^\sigma\subseteq X$ (and more generally with integral extensions, see \cite[TAG 0DYE]{stacks-project}), we obtain
\begin{equation}\label{eq:diagram}
\xymatrix{ \ar[dd]^{E(\iota)}E([\Spec(\sO_{X,x}^h)/G])\ar[rr]^-{\simeq} & & \left( \bigoplus\limits_{\sigma\subseteq G\mbox{ {\small cyclic}}} \widetilde{E}\left( \Spec(\sO^h_{X^\sigma,x})\times\Spec(k[\sigma])\right) \right)^G \ar[dd]^{\oplus_\sigma \widetilde{E}(\iota_\sigma\times\id)} \\ & &  \\ E([\Spec(k)/G])\ar[rr]^-{\simeq} & & \, \left( \bigoplus\limits_{\sigma\subseteq G\mbox{ {\small cyclic}}} \widetilde{E}\left(\Spec(k[\sigma])\right) \right)^G.}
\end{equation}
Here, $\iota_\sigma:\Spec(k)\hookrightarrow\Spec(\sO_{X^\sigma,x}^h)$ and
$\iota:[\Spec(k)/G]\hookrightarrow[\Spec(\sO_{X,x}^h)/G]$ are (induced by) the canonical projection to the residue fields. Since each $\iota_\sigma$ is a closed immersion with
henselian defining ideal, so is each $\iota_\sigma\times\id_{\Spec(k[\sigma])}$, and by
\cite[Theorem A]{CMM}, every map $\widetilde{E}(\iota_\sigma\times\id)$ is an isomorphism, and hence so is $E(\iota)$.

To equate $E(\iota)$ with $\pi_*( \cref{eq:claim})$, and thus to conclude the proof,
it remains to recall that $E(-)=\pi_*(\Kinv(-)/p)$ and that since the order $|G|$ is invertible, 
a finitely generated projective module with a semi-linear $G$-action is the same thing as a finitely generated projective left module over the
twisted group ring, so that we have an equivalence of $\infty$-categories of perfect modules
\[ \mathrm{Perf}([\Spec(\sO_{X,x}^h)/G])\simeq\mathrm{Perf}(\sO_{X,x}^h\int G)\]
and similarly with $\sO^h_{X,x}$ replaced by $k$.
\end{proof}

\section{Nil-invariance, excision and exactness}\label{sec:technicalities}

\subsection{Nil-invariance}
\begin{proposition}\label{prop:nil_inv} Let $G$ be a finite group and $\pi:R\longrightarrow R'$ a surjective
homomorphism of commutative rings with a $G$-action such that $\ker(\pi)$ is nilpotent.
Then $\Kinv(R\int G)\stackrel{\simeq}{\longrightarrow}\Kinv(R'\int G)$ is an equivalence.
\end{proposition}

\begin{proof} This will follow from \cite[Chapter VII, Theorem 0.0.2]{local_structure} if we can show that the kernel of
(the obviously surjective) ring homomorphism $\pi\int G:R\int G\longrightarrow R'\int G$ is nilpotent.
However, an immediate computation shows that for every $n\ge 0$ we have
\[ \left(\ker(\pi\int G)\right)^n\subseteq (\ker(\pi))^n\int G.\]
\end{proof}

\subsection{Excision}

Assume that
\begin{equation}\label{eq:milnor} \xymatrix{ R\ar[rr]\ar[dd]& & S \ar[dd]^g\\ & & \\R'\ar[rr]& & S' }
\end{equation}
is a Milnor square of commutative rings, i.e. a pull-back diagram of rings with $g$ surjective, see
\cite[Chapter IX, \S 5]{bass} for an early account and \cite{tamme_land} for a current development.
If, in addition, a finite group $G$ acts on \cref{eq:milnor}, then the induced square of twisted group rings
\begin{equation}\label{eq:milnorG} \xymatrix{ R\int G \ar[rr]\ar[dd]& & S\int G \ar[dd]^{g\int G}\\ & & \\R'\int G\ar[rr]& & S'\int G }
\end{equation} 
is clearly still a Milnor square. Denoting by $\Kh$ non-connective algebraic $K$-theory and
by $\Khinv$ the fiber of the cyclotomic trace $\Kh\longrightarrow TC$, we then deduce the following
from \cite[Theorem 3.3]{tamme_land}.
\begin{proposition}\label{prop:excision_one} In the above situation, 
\begin{equation}\label{eq:milnorG1} \xymatrix{\Khinv(R\int G) \ar[rr]\ar[dd]& & \Khinv(S\int G) \ar[dd]\\ & & \\\Khinv(R'\int G)\ar[rr]& & \Khinv(S'\int G)}
\end{equation}
is a pull-back square.
\end{proposition}

To pass to connective $K$-theory here, we need the equivariant generalization of
\cite[Corollary 4.34]{CMM}, namely Proposition \ref{prop:eq_rig_K0} below. 

For an associative, unital ring $A$, we denote by \[ \Proj(A)\] the set of isomorphism classes of finitely generated
projective left $A$-modules.



We start by establishing the equivariant generalization of 
\cite[Lemma 4.20]{CMM}:

\begin{proposition}\label{prop:eq_rig_K0}
Let $(R,I)$ be a henselian pair and $G$ a finite group acting on $(R,I)$. Then
the obvious homomorphism
\[ K_0(R\int G)\stackrel{\simeq}{\longrightarrow}K_0((R/I)\int G)\]
is an isomorphism.
\end{proposition}

\begin{proof}
To see that the map is injective, according to \cite[Chapter IX, Proposition 1.3]{bass} it suffices to check
that the kernel of the projection $R\int G\to (R/I)\int G$, namely
\[ I\int G:=\left\{ \sum_{\sigma\in G} a(\sigma)e_\sigma\, \mid \, a(\sigma)\in I\right\}\subseteq R\int G \]
is contained in the radical of $R\int G$. Otherwise, $I\int G$ was not contained in some maximal left ideal
$\n\subseteq R\int G$. Then the subset 
\[ I\int G + \n :=\left\{ x+y \,\mid \, x\in I\int G , y\in \n\right\}\subseteq R\int G\]
was a left-ideal properly containing $\n$, and hence
\begin{equation}\label{eq:equal}
I\int G + \n = R\int G.
\end{equation}
We consider $R=Re_{e}\subseteq R\int G$ as a (non-central !) subring.
Then \cref{eq:equal} holds as an equality of $R$-modules, and since $I\int G=I(R\int G)$ and
$R\int G$ is a finite (and free) $R$-module, Nakayama's lemma\footnote{recall that $I$ is contained in the radical of $R$.} implies that $\n=R\int G$, a contradiction which completes the proof of injectivity.

To see the surjectivity, we will establish the stronger claim that the reduction map
\begin{equation}\label{eq:lift_equivariantly} \xymatrix{ \Proj(R\int G) \ar@{->>}[rr] & & \Proj((R/I)\int G) } \end{equation}
is surjective.
Write $\bar{R}:=R/I$ and fix some $\bar{M}\in\Proj(\bar{R}\int G)$.
We first descend everything to a situation of finite type over the integers.
The ring with $G$-action $R=\cup_\alpha R_\alpha$ is the union of its finitely 
generated, $G$-stable subrings $R_\alpha\subseteq R$. Accordingly, we also have
$\bar{R}=\cup_\alpha(R_\alpha/I_\alpha)=:\cup_\alpha\bar{R}_\alpha$ for $I_\alpha:=R_\alpha\cap I$.
Since then also $\bar{R}\int G=\cup_\alpha\left( \bar{R}_\alpha\int G\right)$, the given $\bar{M}$
descends to some $\bar{M}_\alpha\in\Proj(\bar{R}_\alpha\int G)$ for suitably large indices $\alpha$.

Write $(R_\alpha,I_\alpha)^h$ for the henselization of $R_\alpha$ along $I_\alpha\subseteq R_\alpha$, and note that $G$ naturally acts on $(R_\alpha,I_\alpha)^h$,
because henselization is functorial. It will suffice to lift the given 
$\bar{M}_\alpha\in\Proj(\bar{R}_\alpha\int G)$ to some element of $\Proj((R_\alpha,I_\alpha)^h\int G)$, because the inclusion $R_\alpha\subseteq R$
factors through $(R_\alpha,I_\alpha)^h$ $G$-equivariantly.

We next claim an inclusion of ($G$-invariant) ideals for all sufficiently large $M>>0$, namely
\begin{equation}\label{eq:nil_ideal}
I_\alpha^M\subseteq I^G_\alpha\cdot R_\alpha\subseteq I_\alpha\subseteq R_\alpha.
\end{equation}
Indeed, the second inclusion is obvious, and the first one follows from the 
fact that $I_\alpha$ is finitely generated together with the relation,
valid for every $x\in I_\alpha$:
\[ 0=\prod_{g\in G}(x-g(x))=:x^{|G|}+\sum\limits_{i=0}^{|G|-1}a_ix^i\mbox{ with }a_i\in I_\alpha\cap R_\alpha^G=I^G_\alpha,\]
which implies that $x^{|G|}\in I_\alpha^G\cdot R_\alpha$.

By \cref{eq:nil_ideal}, the kernel of the projection
\[ \xymatrix{ R_\alpha/I_\alpha^G\cdot R_\alpha\simeq R_\alpha^G/I_\alpha^G \otimes_{R_\alpha^G}R_\alpha \ar@{->>}[rr] & & R_\alpha/I_\alpha=\bar{R}_\alpha }\]
is nilpotent, and an easy calculation then shows that so is the kernel of the projection
\[ \xymatrix{ (R_\alpha/I_\alpha^G\cdot R_\alpha)\int G \ar@{->>}[rr] & & \bar{R}_\alpha\int G }.\]
(see the proof of Proposition \ref{prop:nil_inv}).

By \cite[ch. III, Corollary 2.4 and Proposition 2.12]{bass} then, we can lift the given
$\bar{M}_\alpha\in\Proj(\bar{R}_\alpha\int G)$ to some $\bar{M}'_\alpha\in
\Proj((R_\alpha/I_\alpha^G\cdot R_\alpha)\int G)$.

As a final piece of preparation, we need to see what happens to the $G$-invariants under
henselization. Since $R_\alpha^G\subseteq R_\alpha$ is integral, 
and \cref{eq:nil_ideal} shows that $\sqrt{I_\alpha^G\cdot R_\alpha}=
\sqrt{I_\alpha}$, 
\cite[TAG 0DYE]{stacks-project} implies that the canonical map
\begin{equation}\label{eq:base_hens}
(R_\alpha^G,I_\alpha^G)^h\otimes_{R_\alpha^G}R_\alpha\simeq(R_\alpha,I_\alpha)^h
\end{equation}
is an isomorphism.

We are now in a position to lift the given $\bar{M}'_\alpha\in\Proj((R_\alpha/I_\alpha^G\cdot R_\alpha)\int G)$ using
\cite[Theorem 4.1]{greco}, as follows\footnote{The application of this Theorem here is a bit involved because in general neither
is $R\int G$ an $R$-algebra in any obvious way (but only an $R^G$-algebra), nor is
$R^G\subseteq R$ finite.}:

As our henselian pair, we take $(R_\alpha^G,I_\alpha^G)^h$, and as our algebra we
take $A:=(R_\alpha,I_\alpha)^h\int G$: 
The algebra $A$ is a finite $(R^G_\alpha)^h$-module, because it is clearly finite over
$R_\alpha^h$, and \cref{eq:base_hens} shows that $R_\alpha^h$ is finite over
$(R_\alpha^G)^h$, because $R_\alpha^G\subseteq R_\alpha$ is finite, being both 
integral and of finite type.

We then compute the reduction of $A=(R_\alpha,I_\alpha)^h\int G$ to be
\[ \bar{A}:=A/I_\alpha^G\cdot A\simeq (R_\alpha/I_\alpha^G\cdot R_\alpha)\int G.\]
Now, \cite[Theorem 4.1]{greco} shows that the given $\bar{M}_\alpha'\in\Proj((R_\alpha/I_\alpha^G\cdot R_\alpha)\int G)=\Proj(\bar{A})$
lifts to some element of $\Proj(A)=\Proj((R_\alpha,I_\alpha)^h\int G)$, as desired.
\end{proof}


We can now start to work on the version of Proposition \ref{prop:excision_one} for connective
algebraic $K$-theory, at least for those diagrams \cref{eq:milnor} coming from suitable maps of henselian pairs: In the following,
fix henselian pairs $(R,I)$ and $(S,J)$ with an action of the finite group $G$, and assume that $(R,I)\longrightarrow (S,J)$ is a map of pairs which respects the $G$-action and maps $I$
isomorphically to $J$. Then 
\begin{equation}\label{eq:excision_milnor} \xymatrix{ R \ar[rr] \ar[dd] & & R/I\ar[dd]\\ & & \\ S \ar[rr] & & S/J}\end{equation}
is a diagram as in \cref{eq:milnor}, i.e. a Milnor-square with a $G$-action.

We then define $K(R\int G,I\int G)$ by the fiber sequence
\[ K\left(R\int G,I\int G\right)\longrightarrow K\left(R\int G\right)\longrightarrow K\left((R/I)\int G\right),\]
and analogously for $K(S\int G,J\int G)$ and with $K$ replaced by $\Kh$. The map of pairs $(R,I)\to (S,J)$ induces a map \[K\left(R\int G,I\int G\right)\longrightarrow K\left(S\int G,J\int G\right), \]
and similarly for $\Kh$. Recall that there is a 
canonical transformation $K\longrightarrow\Kh$.

\begin{proposition}\label{prop:pull_back}
In the above situation, the diagrams
\begin{itemize}
\item[i)]
\begin{equation}\label{eq:pull_back_con_non_con}
\xymatrix{ K\left(R\int G,I\int G\right) \ar[rr]\ar[dd] & &\ar[dd] K\left(S\int G,J\int G\right)\\ & & \\
\Kh\left(R\int G,I\int G\right) \ar[rr] & & \Kh\left(S\int G,J\int G\right)}
\end{equation}
and
\item[ii)]
\begin{equation}\label{eq:pull_back_con_non_con_inv}
\xymatrix{ \Kinv\left(R\int G,I\int G\right) \ar[rr]\ar[dd] & &\ar[dd] \Kinv\left(S\int G,J\int G\right)\\ & & \\
\Khinv\left(R\int G,I\int G\right) \ar[rr] & & \Khinv\left(S\int G,J\int G\right)}
\end{equation}
\end{itemize}
are pull-back squares.
\end{proposition}

\begin{proof} To prove part $i)$, let $F\longrightarrow \F$ denote the map induced by \cref{eq:pull_back_con_non_con} on horizontal fibres. The claim is that this map is an equivalence.
Since $K\longrightarrow\Kh$ induces an isomorphism on $\pi_k$ for $k\ge 0$, we have $\pi_k(\F)\simeq\pi_k(F)$ for $k\ge 0$. The excision theorem of Milnor-Bass-Murthy \cite[Chapter XII, Theorem 8.3]{bass} applied to the diagram obtained from \cref{eq:excision_milnor} by passing
to twisted group rings shows that $\pi_k(\F)=0$ for all $k\leq -1$. It remains to see that $\pi_k(F)=0$ in this range, too. Since $K_0(R\int G)\stackrel{\simeq}{\longrightarrow}K_0((R/I)\int G)$ is an isomorphism (Proposition \ref{prop:eq_rig_K0}) and the 
maps 
\[ K_1(R\int G)\longrightarrow K_1((R/I)\int G) \mbox{ and }\]
\[ K_1(S\int G)\longrightarrow K_1((S/J)\int G) \]
are surjections (this is true more generally for any surjective ring homomorphism
with kernel contained in the radical, see \cite[ch. IX, Proposition 1.3, (1)]{bass}), we conclude that
the fibers $K(R\int G,I\int G)$ and $K(S\int G,J\int G)$ are concentrated in degrees
$\ge 1$, and thus $F$ is concentrated in degrees $\ge 0$, as claimed.

Part $ii)$ follows from part $i)$ by passage to fibers over $TC$, because the canonical 
transformation $K\to TC$ factors as $K\to\Kh\to TC$.
\end{proof}

\begin{corollary}\label{cor:excision_final} If $(R,I)$ is a henselian pair with a
$G$-action and $R\to S$ is a map of commutative rings with $G$-action mapping
$I$ isomorphically to an ideal $J\subseteq S$, then 
\[ \Kinv(R\int G,I\int G)\stackrel{\simeq}{\longrightarrow} \Kinv(S\int G, J\int G)\]
is an equivalence.
\end{corollary}

\begin{proof} The diagram 
\[ \xymatrix{ R \ar[rr] \ar[dd]& & R/I \ar[dd]\\ & & \\S \ar[rr] & & S/J}\]
is a Milnor-square with $G$-action. Note that the pair $(S,J)$ is also henselian by \cite[Lemma 3.18]{CMM}.
Therefore an application of Proposition \ref{prop:pull_back}, ii) 
reduces our claim to the analogous statement with $\Kinv$ replaced
with $\Khinv$. This is then a special case of Proposition \ref{prop:excision_one}.
\end{proof}

These results will be used to reduce rigidity of arbitrary pairs to rigidity of those 
pairs of the form $(\Z\ltimes I,I)$ already encountered in Section \ref{sec:pscoh}.

\begin{corollary}\label{cor:excision_final1}
For a fixed finite group $G$, there is an equivalence of spectra, functorial in the
henselian pair $(R,I)$ with $G$-action
\[ \Kinv((\Z\ltimes I)\int G, I\int G)\stackrel{\simeq}{\longrightarrow}\Kinv(R\int G,I\int G).\]
\end{corollary}

\subsection{Exactness}

We call a sequence $I'\to I\to\bar{I}$ in $\RingG$ {\em short exact} if it is so when 
considered non-equivariantly, i.e. in $\Ring$, i.e. if the underlying sequence
of abelian groups is short exact (see \cite[Definition 3.4]{CMM}).
We consider the functor
\[ F:\RingG\longrightarrow\Sp,\, F(I):=\Kinv((\Z\ltimes I)\int G, I\int G),\]
and claim that it is exact:

\begin{proposition}\label{prop:exactness}
Given a short exact sequence $I'\to I\to\bar{I}$ in $\RingG$,
then $F(I')\to F(I)\to F(\bar{I})$ is a fiber sequence.
\end{proposition}

\begin{proof}
We consider the following commutative diagram
\[ \xymatrix{ F(I') \ar[dd] \ar[rr] & & F(I) \ar[dd] \ar[rr] & & F(\bar{I}) \ar[dd]\\
& & & & \\ F(I') \ar[rr] & & \Kinv((\Z\ltimes I)\int G) \ar[rr] \ar[dd] & & \ar[dd]\Kinv((\Z\ltimes \bar{I})\int G)\\
& & & & \\ & & \Kinv(\Z[G]) \ar[rr]^= && \Kinv(\Z[G]).}\]

The top row is the one we want to recognize as a fiber sequence. The two
right colomns are the fiber sequences defining $F(I)$ and $F(\bar{I})$. The indicated equality
implies that the upper right square is a pullback. Hence the top row is a fiber sequence if and only if
so is the second row. We verify this by observing that using Corollary \ref{cor:excision_final}
for the obvious map of pairs with $G$-action $(\Z\ltimes I',I')\to (\Z\ltimes I, I')$ gives
\[ F(I')\stackrel{def}{=}\Kinv((\Z\ltimes I')\int G, I'\int G)\stackrel{\simeq}{\longrightarrow}
\Kinv((\Z\ltimes I)\int G, I'\int G)\]
\[ \stackrel{def}{=}\mathrm{fiber}\left(\Kinv((\Z\ltimes I)\int G)\longrightarrow \Kinv(((\Z\ltimes I)/I')\int G)\right).\]
Since $(\Z\ltimes I)/I'\simeq \Z\ltimes\bar{I}$, this concludes the proof.
\end{proof}

\section{The proof of the main result}\label{sec:proof_of_main}

In this section, we give the proof of our main result, Theorem \ref{thm:main}, which we restate 
for convenience.

\begin{theorem}
If the finite group $G$ acts on the henselian pair $(R,I)$, $|G|\in R^*$, and $n\ge 1$ is an integer coprime to $|G|$, then the reduction map
\[ \Kinv(R\int G)/n\stackrel{\simeq}{\longrightarrow} \Kinv((R/I)\int G)/n\]
is an equivalence.
\end{theorem}

\begin{proof}
We can assume that $n=p$ is a prime (not dividing $|G|$).
Since $\Kinv(R\int G,I\int G)/p\simeq \Kinv((\Z\ltimes I)\int G, I\int G)/p$ (see Corollary \ref{cor:excision_final1}),
our claim is that the functor
\[ F:\RingG_{\Z\left[\frac{1}{|G|}\right]}\longrightarrow\Sp, \, F(I):=\Kinv((\Z\ltimes I)\int G,I\int G)/p\]
is trivial. We start by collecting properties of $F$ that were established previously:
By Proposition \ref{prop:G_pscoh_final}, 
\begin{equation}\label{eq:pscoh} F \mbox{ is } G\mbox{-pscoh.}
\end{equation}
By Proposition \ref{prop:exactness},
\begin{equation}\label{eq:exact} F \mbox{ sends short exact sequences to fiber sequences}
\end{equation}
and by Proposition \ref{prop:nil_inv}, 
\begin{equation}\label{eq:nil} F \mbox{ vanishes on nilpotent arguments.} 
\end{equation}

For every prime field $\Omega$
of characteristic not dividing $|G|$\footnote{by convention, this is fulfilled for characteristic zero.}
recall the compact projective generators $F''_\Omega(n)\in\RingG_\Omega$ $(n\ge 0)$ from
Proposition \ref{prop:equivariant_generators}. We deduce from Proposition \ref{prop:geometric_rigidity}
\footnote{This is the step which forces us to assume that $p$ does not divide $|G|$, and that the characteristic of $\Omega$ does not divide $|G|$.}
that $F(F''_\Omega(n))=0$ for all $n\ge 0$. Since by \cref{eq:pscoh}, the restriction
of $F$ to $\RingG_\Omega$ is $G$-pscoh, and in particular left Kan extended from its
subcategory of compact projective objects (see \cite[Lemma 4.6]{CMM}), which in turn is the idempotent completion of all the 
$F''_\Omega(n)$, we see that for every prime field $\Omega$ of characteristic not dividing $|G|$, we have
\begin{equation}\label{eq:vanish_field} F(\RingG_\Omega)=0.
\end{equation}

We now boot-strap to see that for every $N\ge 1$, $I\in \RingG_{\Z\left[\frac{1}{|G|}\right]}$:
\begin{equation}\label{eq:vanish_torsion} \mbox{ if } (N,|G|)=1\mbox{ and } NI=0, \mbox{ then } F(I)=0.
\end{equation}

Since $F$ preserves finite products, we can assume that $N=q^r$ is a prime-power
(with the prime $q$ not dividing $|G|$) and then consideration of the short exact sequence $qI\to I\to I/qI$
together with \cref{eq:exact}, \cref{eq:nil} and \cref{eq:vanish_field} (for $\Omega=\mathbb{F}_q$) proves \cref{eq:vanish_torsion}.

Since $F$ is bounded below, there is an integer $d\in\Z$ such that
\begin{equation}\label{eq:minimality} \pi_k F =0, \mbox{ for every } k<d.
\end{equation}

We will be done if we can show that the functor to abelian groups
$F_0:=\pi_d F:\RingG_{\Z\left[\frac{1}{|G|}\right]}\longrightarrow\mathrm{Ab}$
vanishes, because $d$ being arbitrary, this will imply that $F=0$.
To see this, we will establish the following:
\begin{equation}\label{eq:universal_torsion} \mbox{ There is some } N \mbox{ coprime to } |G|\mbox{ such that for all } I\in\RingG_{\Z\left[\frac{1}{|G|}\right]},\end{equation}
\begin{equation*}F_0(NI)\longrightarrow F_0(I)\mbox{ is the zero map.}\end{equation*}

Given this, using \cref{eq:exact} and \cref{eq:minimality}, we obtain an exact sequence
\[ F_0(NI)\stackrel{0}{\longrightarrow} F_0(I)\longrightarrow F_0(I/NI)\longrightarrow 0,\]
where $F_0(I/NI)=0$ by \cref{eq:vanish_torsion},
hence $F_0(I)=0$.

To prove \cref{eq:universal_torsion}, we recall (Proposition \ref{prop:equivarint_gabber_colim}) that we have a relation between free objects
\[F''_{\mathbb{Q}}(n)=\mathrm{colim}_{(N,|G|)=1} F''_{\Z\left[\frac{1}{|G|}\right]}(n).\]

Since $F_0(F''_{\Q}(n))=0$ by \cref{eq:vanish_field} for $\Omega=\mathbb{Q}$
and $F_0$ commutes with filtered colimits, we deduce that for every $x\in F_0({F''_{\Z\left[\frac{1}{|G|}\right]}(n)})$ there is some $N$ coprime to $|G|$ 
(depending on $x$ and $n$) such that $[N](x)=0$. To deduce from this the more uniform
statement \cref{eq:universal_torsion} one uses that $F_0$ is finitely generated and takes
the product of all $N$'s for the generators. Since the details of this step are literally the same
as in the proof of \cite[Lemma 4.16]{CMM}, we omit them here.
\end{proof}

\bibliography{equivariant_rigidity}
\bibliographystyle{alpha}

\end{document}